\documentclass[11pt, reqno]{amsart}
\usepackage{amsmath, amsthm, amscd, amsfonts, amssymb, graphicx, color}
\usepackage[bookmarksnumbered, colorlinks, plainpages]{hyperref}
\usepackage{amsmath}
\usepackage{amssymb}
\usepackage{amsthm}
\usepackage{enumerate}
\usepackage[mathscr]{eucal}
\theoremstyle{plain}
\newtheorem{thm}{Theorem}[section]

\newtheorem{Example}{Example}[section]
\newtheorem{note}{Note}[section]

\theoremstyle{definition}
\newtheorem{defn}{Definition}[section]
\newtheorem{rem}{Remark}[section]
\numberwithin{equation}{section}

\begin{document}

\setcounter {page}{1}
\title[$\mathcal{I}$-convergence in metric-like spaces]
{ $\mathcal{I}$-convergence of sequences in metric-like spaces }

\author[ P. Malik, S. Das ]{Prasanta Malik* and Saikat Das*\ }
\newcommand{\acr}{\newline\indent}
\maketitle
\address{{*\,} Department of Mathematics, The University of Burdwan, Golapbag, Burdwan-713104,
West Bengal, India.
                Email: pmjupm@yahoo.co.in, dassaikatsayhi@gmail.com \acr
           }

\maketitle


\textbf{keywords}: Ideal, filter, $\mathcal{I}$-convergence, partial metric space, metric-like space.

\begin{abstract}
In this paper we introduce and study the notion of $\mathcal{I}$-convergence of sequences in 
a metric-like space, where $\mathcal{I}$ is an ideal of subsets of the set $\mathbb{N}$ of all 
natural numbers. Further introducing the notion of $\mathcal{I}^{*}$-convergence 
of sequences in a metric-like space we study its relationship with $\mathcal{I}$-convergence.
\end{abstract}

\section{\bf{Introduction and background}}

The notion of convergence of sequences of real numbers was extended to the 
notion of statistical convergence by Fast \cite{Fa} (and also independently by Schoenberg \cite{Sc} ) with the 
concept of natural density. A lot of work have been done in this direction, after the work of 
Salat \cite{Sl} and Fridy \cite{Fr1}. For more primary works in this line one can see  \cite{ Fr2, Ko1} etc.

The notion of statistical convergence of sequences of real numbers further extended to the notion
of $\mathcal{I}$-convergence by Kostyrko et al. \cite{Ko2}, with the notion of an ideal $\mathcal{I}$ of subsets 
of the set of all natural numbers $\mathbb{N}$. For more works in this direction one can see \cite{ Das, Dm, Lahiri1, Lahiri2}.

On the other hand the concept of partial metric space was first introduced by Matthews \cite{Matt},
as a generalization of the usual notion of metric space. In \cite{Am}, A. A. Harandi introduced the concept of metric-like space which is a 
generalization of the concepts of metric space as well as partial metric space and studied the 
notions of convergence and Cauchyness of sequences in a metric-like space. For more works on partial metric spaces and metric-like spaces one can see
\cite{Al, Buk, Hoss, Ol}.

In \cite{Nu} F. Nuray introduced and studied the notion of statistical convergence in a partial metric space
and in \cite{Gull} E. G$\ddot{u}$lle et al. extended this notion to $\mathcal{I}$-convergence in partial
metric setting. Recently in \cite{Malik} Malik et al. has introduced and studied the notion of statistical 
convergence of sequences in a metric-like space. It seems therefore reasonable to introduce
and study the notion of $\mathcal{I}$-convergence of sequences in a metric-like space. In this paper we do the same
and studied some basic properties of $\mathcal{I}$-convergence of sequences in a metric-like space. Further introducing 
the notions of $\mathcal{I}^{*}$-convergence in a metric-like space we have examined its relationship with $\mathcal{I}$
-convergence.    


\section{\bf{Basic Definitions and Notations}}

In this section we recall some basic definitions and notations from the literature, which is required
throughout our study. Throughout this paper $\mathbb{R}$ denotes the set of all real numbers,
$\mathbb{R}_{\geq 0}$ denotes the set of all non-negative real numbers.

\begin{defn} \cite{Matt} \label{ defn 1}
Let $\mathcal{X}$ be a non-empty set and a mapping $p: \mathcal{X \times X} \rightarrow \mathbb{R}_{\geq 0}$ is
said to be a partial metric on $\mathcal{X}$ if for any $x, y, z \in \mathcal{X}$, the following conditions
are satisfied:
\begin{itemize}
\item[$(p1)$] $x = y$ if and only if $p(x,x) = p(x,y) = p(y,y)$;
\item[$(p2)$] $0 \leq p(x,x) \leq p(x,y)$;
\item[$(p3)$] $p(x,y) = p(y,x)$;
\item[$(p4)$] $p(x,y) \leq p(x,z) + p(z,y) - p(z,z)$. 
\end{itemize}

The pair $(\mathcal{X},p)$ is then called a partial metric space.
\end{defn}

\begin{defn} \cite{Am} \label{ defn 2}
Let $\mathcal{X}$ be a non-empty set. A mapping $\delta: \mathcal{X \times X}\rightarrow \mathbb{R}_{\geq 0}$
is said to be a metric-like on $\mathcal{X}$ if for any $x, y, z \in \mathcal{X}$, the following conditions
are satisfied:
\begin{itemize}
\item[$(\delta1)$] $\delta(x,y) = 0 ~\Rightarrow~ x = y$;
\item[$(\delta2)$] $\delta(x,y) = \delta(y,x)$;
\item[$(\delta3)$] $\delta(x,y) \leq \delta(x,z) + \delta(z,y)$.
\end{itemize}  

The pair $(\mathcal{X},\delta)$ is then called a metric-like space. Note that in a metric-like space $(\mathcal{X}
,\delta), ~\delta(x,x)$ may be positive for some $x \in \mathcal{X}$.
\end{defn}

\begin{rem}
We see from Definition \ref{ defn 1} and Definition \ref{ defn 2} that every metric space is a partial metric
space and every partial metric space is a metric-like space, but the converses are not true. 
\end{rem}

\begin{defn} \cite{Am}:
Let $x_{o}$ be a point in a metric-like space $(\mathcal X,\delta)$ and let $\epsilon > 0$.
Then the open $\delta$-ball with centered at $x_{o}$ and radius $\epsilon > 0$ in $(\mathcal X,\delta)$
is denoted by $\mathcal{B}_{\delta}(x_{o};\epsilon)$ and is defined by
\begin{align*}
\mathcal{B}_{\delta}(x_{o};\epsilon)= \{ x \in \mathcal{X}: \left|\delta(x,x_{o})- 
\delta(x_{o},x_{o})\right|< \epsilon \}.
\end{align*} 
\end{defn}

\begin{defn} \cite{Am}:
A sequence $\{x_{n}\}_{n\in \mathbb N}$ in a metric-like space $(\mathcal X,\delta)$ is said to be 
convergent to a point $x_{o} (\in \mathcal X)$, if for every $\epsilon > 0$, there exists $k_{o} 
\in \mathbb N$ such that
\begin{align*}
&\left|\delta(x_{n}, x_{o})- \delta(x_{o}, x_{o})\right|< \epsilon,  &\forall~ n\geq k_{o}\\
\mbox{i.e.}~ & x_{n} \in \mathcal{B}_{\delta}(x_{o};\epsilon),  &\forall~ n\geq k_{o}.
\end{align*}
In this case, we write $\displaystyle{\lim_{n\rightarrow \infty}}x_{n} = x_{o}$.
\end{defn}

We now recall the concept of statistical convergence and $\mathcal{I}$-convergence of real sequences.

\begin{defn} \cite{Fr1}:
Let $\mathcal P$ be a subset of $\mathbb N$. For each $n \in \mathbb{N}$, let $\mathcal{P}(n)$ 
denote the cardinality of the set $\{k \leq n: k \in \mathcal{P}\}$. We say that the set $\mathcal{P}$ 
has natural density $d(\mathcal{P})$ if the limit $\displaystyle{\lim_{n\rightarrow \infty}}\mathcal{P}(n)/n$
exists finitely and
\begin{align*}
d(\mathcal{P}) = \displaystyle{\lim_{n\rightarrow \infty}} \frac{\mathcal{P}(n)}{n}.
\end{align*}
\end{defn}

\begin{defn} \cite{Fr1}:
Let $\{x_{n}\}_{n\in \mathbb N}$ be a sequence of real numbers. Then $\{x_{n}\}_{n\in \mathbb N}$
is said to be statistically convergent to $x_{o} (\in \mathbb R)$, if for any $\epsilon > 0,~ d(\mathcal
A(\epsilon ))= 0$, where
\begin{align*}
\mathcal A(\epsilon)= \left\{n\in \mathbb N: \left|x_{n} - x_{o}\right|\geq \epsilon\right\}.
\end{align*}
In this case we write $st-\displaystyle{\lim_{n\rightarrow \infty}}x_{n}= x_{o}.$
\end{defn}

\begin{defn} \cite{Ko2}: 
Let $\mathcal{D}$ be a non-empty set. A non-empty class $\mathcal I$ of subsets of $\mathcal D$ is said to be an 
ideal in $\mathcal{D}$, provided $\mathcal I$ satisfies the following conditions:
\begin{itemize}
\item[$(i)$] $\phi \in\mathcal I$, 
\item[$(ii)$] if $\mathcal{A}, \mathcal{B} \in \mathcal I$ then $\mathcal{A}\cup \mathcal{B} \in\mathcal I$, and 
\item[$(iii)$] if $\mathcal{A}\in \mathcal I$ and $\mathcal{B}\subset \mathcal{A}$ then $\mathcal{B}\in \mathcal I$.
\end{itemize}
\end{defn}
 
An ideal $\mathcal{I}$ in $\mathcal{D}~ (\neq \phi)$ is called non-trivial if $\mathcal{D} \notin \mathcal{I}$ 
and $\mathcal{I} \neq \{\phi\}$.

A non-trivial ideal $\mathcal{I}$ in $\mathcal{D}~ (\neq \phi)$ is called an admissible ideal if 
$\{z\} \in \mathcal{I}, ~\forall~ z \in\mathcal{D}$. 

Throughout the paper we take $\mathcal I$ as an admissible ideal of $\mathbb{N}$, unless otherwise mentioned.
 
\begin{defn} \cite{Ko2}: 
A non-empty class $\mathcal{F}$ of subsets of $\mathcal{D}~ (\neq \phi)$ is said to be a filter 
in $\mathcal{D}$, provided 
\begin{itemize}
\item[$(i)$] $\phi \notin \mathcal F$, 
\item[$(ii)$] if $\mathcal{A}, \mathcal{B} \in\mathcal{F}$ 
then $\mathcal{A} \cap \mathcal{B} \in\mathcal{F}$, and
\item[$(iii)$] if $\mathcal{A} \in \mathcal{F}$ and $\mathcal{B}$ is a subset of $\mathcal{D}$ such that 
$\mathcal{A} \subset \mathcal{B}$ then $\mathcal{B} \in \mathcal{F}$.
\end{itemize}
\end{defn}
  
Let $\mathcal{I}$ be a non-trivial ideal in $\mathcal{D}$. Then $\mathcal{F(I)}= \{ \mathcal{D-A} : \mathcal{A} \in\mathcal{I}\}$
forms a filter on $\mathcal{D}$, called the filter associated with the ideal $\mathcal{I}$.
 
\begin{defn} \cite{Ko2}
Let $\mathcal{I}$ be an admissible ideal in $\mathbb{N}$. Then $\mathcal{I}$ is said to satisfy the condition
\textit{(AP)} if for every countable collection of mutually disjoint sets $\{A_{j}\}_{j\in \mathbb{N}}$ belonging
to $\mathcal{I}$, there exists a countable family of sets $\{B_{j}\}_{j\in \mathbb{N}}$ such that $A_{j} \Delta B_{j}$
is a finite set for each $j\in \mathbb{N}$ and $B = \bigcup^{\infty}_{j=1} B_{j} \in \mathcal{I}$.
\end{defn}

\begin{defn} \cite{Ko2}:
A sequence $\{x_{n}\}_{n\in \mathbb{N}}$ of real numbers is said to be $\mathcal{I}$-convergent to $x \in \mathbb{R}$ if 
$\forall~ \epsilon > 0$, the set $\mathcal{A}(\epsilon)= \{n \in\mathbb{N}: \left| x_{n}- x \right|\geq \epsilon\} \in\mathcal{I}$
or in other words for each $\epsilon > 0, ~\exists ~\mathcal{B}(\epsilon)\in \mathcal{F(I)}$ such that
\begin{align*}
\left|x_{n}- x\right| < \epsilon, ~\forall~ n \in\mathcal{B}(\epsilon). 
\end{align*}

In this case we write $\mathcal{I}-\displaystyle{\lim_{n\rightarrow \infty}}x_{n}= x$.
\end{defn}

Following \cite{Malik} we now recall the notion of statistical convergence of a sequence $\{x_{n}\}_{n\in \mathbb{N}}$
in a metric-like space $(\mathcal{X},\delta)$.

\begin{defn} \cite{Malik}
Let $(\mathcal{X},\delta)$ be a metric-like space and let $\{x_{n}\}_{n\in \mathbb{N}}$ be a sequence in $\mathcal{X}$.
Then $\{x_{n}\}_{n\in \mathbb{N}}$ is said to be statistically convergent to a $x_{o} ~(\in \mathcal{X})$ if 
$st-\displaystyle{\lim_{n\rightarrow \infty}}\delta(x_{n},x_{o}) = \delta(x_{o},x_{o})$ i.e. for any
$\epsilon > 0$,
\begin{align*}
& d(\{n\in \mathbb{N}: \left|\delta(x_{n},x_{o}) - \delta(x_{o},x_{o})\right| \geq \epsilon\}) = 0 \\
\mbox{i.e.}~ & d(\{n\in \mathbb{N}: x_{n} \notin \mathcal{B}_{\delta}(x_{o};\epsilon)\}) = 0.
\end{align*}
\end{defn}


\section{\bf{$\mathcal{I}$-Convergence }}

In this section we introduce and study the notion of $\mathcal{I}$-convergence of sequences in
a metric-like space.

\begin{defn}
Let $(\mathcal{X},\delta)$ be a metric-like space and $\{x_{n}\}_{n\in \mathbb{N}}$ be a sequence in
$\mathcal{X}$. Then $\{x_{n}\}_{n\in \mathbb{N}}$ is said to be $\mathcal{I}$-convergent to $x_{o} \in
\mathcal{X}$, if $\mathcal{I}-\displaystyle{\lim_{n\rightarrow \infty}}\delta(x_{n},x_{o}) = \delta(x_{o},x_{o})$
i.e. for every $\epsilon > 0$, the set
\begin{align*}
\mathcal{A}(\epsilon) = \{n\in \mathbb{N}: \left|\delta(x_{n},x_{o}) - \delta(x_{o},x_{o})\right| \geq \epsilon\}
\in \mathcal{I}.
\end{align*}
In other words, for every $\epsilon > 0, ~\exists~ \mathcal{B}(\epsilon) \in \mathcal{F(I)}$ such that
\begin{align*}
\left|\delta(x_{n},x_{o}) - \delta(x_{o},x_{o})\right| < \epsilon, ~\forall~ n \in \mathcal{B}(\epsilon).
\end{align*}

In this case we write $\mathcal{I}-\displaystyle{\lim_{n\rightarrow \infty}}x_{n} = x_{o}$ and $x_{o}$
is called an $\mathcal{I}$-limit of $\{x_{n}\}_{n\in \mathbb{N}}$.
\end{defn}

We now cite an example, to show that $\mathcal{I}$-limit of a sequence in a metric-like space
may not be unique. 

\begin{Example} \label{ Eaxmp 1}
Let us consider a decomposition of $\mathbb{N}$ as $\mathbb{N} = \bigcup^{\infty}_{j=1} \mathcal{D}_{j}$,
where every $\mathcal{D}_{j}$ is an infinite subset of $\mathbb{N}$ and $\mathcal{D}_{i} \cap \mathcal{D}_{j}
= \phi$ for $i \neq j$. Let $\mathcal{I} = \{A \subset \mathbb{N}: A ~\mbox{intersects only finitely many}~
\mathcal{D}_{j} ~'s\}$. Then $\mathcal{I}$ is an admissible ideal of $\mathbb{N}$.

Let $\mathcal{X}$ denotes the set of all non-negative real numbers and $\delta: \mathcal{X \times X} \rightarrow
\mathbb{R}_{\geq 0}$ be defined as follows:
\begin{align*}
\delta(x,y) = \begin{cases}
0, & \mbox{if}~ x = y ~\mbox{and x is irrational in}~ \mathcal{X}  \\
2, & \mbox{if}~ x, y ~\mbox{are both integers in}~ \mathcal{X} \\
3, & \mbox{otherwise}.
\end{cases}
\end{align*}
Then $(\mathcal{X},\delta)$ is a metric-like space but neither a partial metric space
nor a metric space. We now define a sequence $\{x_{n}\}_{n\in \mathbb{N}}$ in $\mathcal{X}$ 
as follows:
\begin{align*}
x_{n} = \begin{cases}
\frac{1}{2}, & \mbox{if}~ n \in \mathcal{D}_{2} \\
j, & \mbox{if}~ n \in \mathcal{D}_{j} ~(j \in \mathbb{N}- \{2\}).
\end{cases}
\end{align*}
Then
\begin{align*}
\left|\delta(x_{n},1) - \delta(1,1)\right| = \begin{cases}
1, & \mbox{if}~ n \in \mathcal{D}_{2} \\
0, & \mbox{if}~ n \in \mathcal{D}_{j} ~(j \in \mathbb{N} - \{2\}).
\end{cases}
\end{align*}
Let $\epsilon > 0$ be given. Then
\begin{align*}
\{n \in \mathbb{N}: \left|\delta(x_{n},1) - \delta(1,1)\right| \geq \epsilon\} = \begin{cases}
\mathcal{D}_{2}, & \mbox{if}~ 0 < \epsilon \leq 1 \\
\phi, & \mbox{if}~ \epsilon > 1.
\end{cases}
\end{align*}
This implies $\{n \in \mathbb{N}: \left|\delta(x_{n},1) - \delta(1,1)\right| \geq \epsilon\} \in \mathcal{I}$
and hence $\mathcal{I}-\displaystyle{\lim_{n\rightarrow \infty}} \delta(x_{n},1) = \delta(1,1)$. So we have
$\mathcal{I}-\displaystyle{\lim_{n\rightarrow \infty}} x_{n} = 1$.

Again we have,
\begin{align*}
\left|\delta(x_{n},2) - \delta(2,2)\right| = \begin{cases}
1, & \mbox{if}~ n \in \mathcal{D}_{2} \\
0, & \mbox{if}~ n \in \mathcal{D}_{j} ~(j \in \mathbb{N} - \{2\}).
\end{cases}
\end{align*}
So, in a similar way as above,
\begin{align*}
& \{n \in \mathbb{N}: \left|\delta(x_{n},2) - \delta(2,2)\right| \geq \epsilon\} \in \mathcal{I} \\
\Rightarrow~ & \mathcal{I}-\displaystyle{\lim_{n\rightarrow \infty}} \delta(x_{n},2) = \delta(2,2) \\
\Rightarrow~ & \mathcal{I}-\displaystyle{\lim_{n\rightarrow \infty}} x_{n} = 2.
\end{align*}
Thus we see that $\mathcal{I}$-limit of the sequence $\{x_{n}\}_{n\in \mathbb{N}}$ in $(\mathcal{X},\delta)$
is not unique in a metric-like space which is different from the result in a metric space (see Proposition 3.1 \cite{Ko2}). 
\end{Example}

\begin{rem}
(i). If $\mathcal{I} = \mathcal{I}_{f} = \{A \subset \mathbb{N}: A ~\mbox{is finite}\}$, then $\mathcal{I}_{f}$
is an admissible ideal of $\mathbb{N}$ and $\mathcal{I}_{f}$-convergence of a sequence in a metric-like space coincides
with the usual convergence as introduced by A. A. Harandi \cite{Am}. 

(ii). If $\mathcal{I} = \mathcal{I}_{d} = \{A \subset \mathbb{N}: d(A) = 0\}$, then $\mathcal{I}_{d}$ is an
admissible ideal of $\mathbb{N}$ and $\mathcal{I}_{d}$-convergence of a sequence in a metric-like space coincides
with the statistical convergence as introduced by Malik et al. \cite{Malik}. 
\end{rem}

\begin{thm} \label{ Theo 1}
Let $(\mathcal{X},\delta)$ be a metric-like space and $\{x_{n}\}_{n\in \mathbb{N}}$ be a sequence in
$\mathcal{X}$. If $\{x_{n}\}_{n\in \mathbb{N}}$ converges to $x_{o}$ in $\mathcal{X}$
then it is $\mathcal{I}$-convergent to the same limit.
\end{thm}

\begin{proof}
Proof is straightforward so is omitted.
\end{proof}

\begin{note}
The converse of the Theorem \ref{ Theo 1} is not true. 
\end{note}

\begin{thm} \label{ Theo 2}
Let $(\mathcal{X},\delta)$ be a metric-like space and $\{x_{n}\}_{n\in \mathbb{N}}$ be a sequence in $\mathcal{X}$.
If $\{x_{n}\}_{n\in \mathbb{N}}$ is $\mathcal{I}$-convergent to $x_{o} \in \mathcal{X}$, then there exists a subsequence
of $\{x_{n}\}_{n\in \mathbb{N}}$, which converges to $x_{o}$ in usual sense.
\end{thm}

\begin{proof}
For every $j \in \mathbb{N}$ we let $A_{j} = \{n\in \mathbb{N}: \left|\delta(x_{n},x_{o}) - \delta(x_{o},x_{o})\right|
< \frac{1}{j}\}$. As $\mathcal{I}-\displaystyle{\lim_{n\rightarrow \infty}} \delta(x_{n},x_{o}) = \delta(x_{o},x_{o})$
and $\mathcal{I}$ is an admissible ideal of $\mathbb{N}$, so $A_{j}~'s$ are infinite subset of $\mathbb{N}$ and
$A_{1} \supset A_{2} \supset A_{3} \ldots$.

Let $n_{1}$ be the least element of $A_{1}$. As $A_{2}$ is an infinite set, so there exists $n_{2} \in A_{2}$ with 
$n_{1} < n_{2}$. Similarly there exists $n_{3} \in A_{3}$ such that $n_{1} < n_{2} < n_{3}$. Proceeding in this way
we will obtained a strictly increasing sequence $\{n_{k}\}_{k\in \mathbb{N}}$ of natural numbers such that $\{x_{n_{k}}\}
_{k\in \mathbb{N}}$ is a subsequence of $\{x_{n}\}_{n\in \mathbb{N}}$ and 
\begin{align} \label{ eq 3}
\left|\delta(x_{n_{k}},x_{o}) - \delta(x_{o},x_{o})\right| < \frac{1}{k}, ~\forall~ k \in \mathbb{N}.
\end{align}
Let $\epsilon > 0$ be given. Then there exists $k_{o} \in \mathbb{N}$ such that $\frac{1}{k_{o}} < \epsilon$. So
from (\ref{ eq 3}) we have,
\begin{align*}
& \left|\delta(x_{n_{k}},x_{o}) - \delta(x_{o},x_{o})\right| < \frac{1}{k} \leq \frac{1}{k_{o}} < \epsilon,
~\forall~ k \geq k_{o} \\
\Rightarrow~ & \displaystyle{\lim_{k\rightarrow \infty}} \delta(x_{n_{k}},x_{o}) = \delta(x_{o},x_{o}) \\
\Rightarrow~ & \displaystyle{\lim_{k\rightarrow \infty}} x_{n_{k}} = x_{o}.
\end{align*}
\end{proof}

\begin{rem}
The converse of the Theorem \ref{ Theo 2} is not true. 
\end{rem}

Now a natural question arises, under which condition the converse of the Theorem \ref{ Theo 2} holds.
The following theorem gives an affirmative answer of this question.

\begin{thm} \label{ Theo 2.A}
Let $(\mathcal{X},\delta)$ be a metric-like space and $\{x_{n}\}_{n\in \mathbb{N}}$ be a sequence in $\mathcal{X}$.
If each subsequence of $\{x_{n}\}_{n\in \mathbb{N}}$ has a subsequence, which is $\mathcal{I}$-converges to $x_{o}
\in \mathcal{X}$, then the sequence $\{x_{n}\}_{n\in \mathbb{N}}$ is $\mathcal{I}$-converges to $x_{o}$.
\end{thm} 

\begin{proof}
Let $\{x_{n}\}_{n\in \mathbb{N}}$ be a sequence in $\mathcal{X}$ such that each subsequence of
$\{x_{n}\}_{n\in \mathbb{N}}$ is $\mathcal{I}$-convergent to $x_{o} \in \mathcal{X}$.
If possible, let $\{x_{n}\}_{n\in \mathbb{N}}$ be not $\mathcal{I}$-convergent to $x_{o}$.
Then there exists $\epsilon_{o} > 0$ such that 
\begin{align*}
\mathcal{A}(\epsilon_{o}) = \{n\in \mathbb{N}: \left|\delta(x_{n},x_{o}) - \delta(x_{o},x_{o})\right| \geq \epsilon_{o}\}
\notin \mathcal{I}.
\end{align*}
As $\mathcal{I}$ is an admissible ideal of $\mathbb{N}$, so $\mathcal{A}(\epsilon_{o})$ must be an infinite
subset of $\mathbb{N}$. We write $\mathcal{A}(\epsilon)$ as $\mathcal{A}(\epsilon_{o}) = \{n_{1} < n_{2} < \ldots\}$. 
Then $\{x_{n_{k}}\}_{k\in \mathbb{N}}$ is a subsequence of $\{x_{n}\}_{n\in \mathbb{N}}$ and
\begin{align} \label{ eq 4}
\left|\delta(x_{n_{k}},x_{o}) - \delta(x_{o},x_{o})\right| \geq \epsilon_{o}, ~\forall~ k \in \mathbb{N}.
\end{align}
Let $y_{p} = x_{n_{k_{p}}}, ~\forall~ p \in \mathbb{N}$. Then $\{y_{p}\}_{p\in \mathbb{N}}$ is a subsequence of
$\{x_{n_{k}}\}_{k\in \mathbb{N}}$ and so from (\ref{ eq 4}) we have,
\begin{align*}
\{p \in \mathbb{N}: \left|\delta(y_{p},x_{o}) - \delta(x_{o},x_{o})\right| \geq \epsilon_{o}\} = \mathbb{N}.
\end{align*} 
So, $\{p\in \mathbb{N}: \left|\delta(y_{p},x_{o}) - \delta(x_{o},x_{o})\right| \geq \epsilon_{o}\} \notin \mathcal{I}$,
as $\mathcal{I}$ is an admissible ideal. This implies that $\{y_{p}\}_{p\in \mathbb{N}}$ is not $\mathcal{I}$-convergent
to $x_{o}$, which contradicts our hypothesis. Therefore $\{x_{n}\}_{n\in \mathbb{N}}$ must be 
$\mathcal{I}$-convergent to $x_{o}$.
\end{proof}

\begin{rem}
Converse of the Theorem \ref{ Theo 2.A} is not true. 
\end{rem}


\section{ \bf{$\mathcal{I}^{*}$-convergence } }

In this section we introduce and study the notion of $\mathcal{I}^{*}$-convergence in a metric-like
space. For this study we need some basic topological properties of a metric-like space as introduced below.

\begin{defn}
Let $(\mathcal{X},\delta)$ be a metric-like space and $A$ be a non-empty subset of $\mathcal{X}$. A
point $x_{o} \in \mathcal{X}$ is said to be a limit point of $A$, if for each $r > 0$,
\begin{align*}
\mathcal{B}_{\delta}(x_{o};r) \cap (A - \{x_{o}\}) \neq \phi.
\end{align*}
\end{defn}

\begin{defn}
A point $x_{o}$ in a metric-like space $(\mathcal{X},\delta)$ is said to be an isolated point
of $\mathcal{X}$, if there exists $r_{o} > 0$ such that
\begin{align*}
\mathcal{B}_{\delta}(x_{o};r_{o}) = \{x_{o}\}.
\end{align*}
\end{defn}

\begin{defn}
A metric-like space $(\mathcal{X},\delta)$ is said to be a $T_{0}$ metric-like space or is said to
satisfy the $T_{0}$ axiom, if for every pair of distinct points in $\mathcal{X}$, there exists an
open $\delta$-ball, which contains one of the points but not the other.
\end{defn}

Now we introduce the notion of $\mathcal{I}^{*}$-convergence in a metric-like space.

\begin{defn}
Let $\mathcal{I}$ be an admissible ideal of subsets of $\mathbb{N}$. A sequence $\{x_{n}\}_{n\in \mathbb{N}}$ of
points in a metric-like space $(\mathcal{X},\delta)$ is said to be $\mathcal{I}^{*}$-convergent to $x_{o} \in 
\mathcal{X}$ if there exists $M = \{n_{1} < n_{2} < \ldots\} \in \mathcal{F(I)}$ such that $\displaystyle{\lim
_{k\rightarrow \infty}} \delta(x_{n_{k}},x_{o}) = \delta(x_{o},x_{o})$. 

In this case we write $\mathcal{I}^{*}-\displaystyle{\lim_{n\rightarrow \infty}} x_{n} = x_{o}$ and $x_{o}$ is called
an $\mathcal{I}^{*}$-limit of $\{x_{n}\}_{n\in \mathbb{N}}$.
\end{defn}

\begin{thm} \label{ Theo 4}
Let $\mathcal{I}$ be an admissible ideal in $\mathbb{N}$, $\{x_{n}\}_{n\in \mathbb{N}}$ be a 
sequence in a metric-like space $(\mathcal{X},\delta)$ and $x_{o} \in \mathcal{X}$. If 
$\mathcal{I}^{*}-\displaystyle{\lim_{n\rightarrow \infty}} x_{n} = x_{o}$, 
then $\mathcal{I}-\displaystyle{\lim_{n\rightarrow \infty}} x_{n} = x_{o}$.
\end{thm}

\begin{proof}
Proof is straightforward, so is omitted
\end{proof}

\begin{rem}
The converse of the above theorem is not true in general, it depends on the structure of the metric-like
space as seen in the subsequent results.
\end{rem}

\begin{thm} \label{ Theo 5}
Let $(\mathcal{X},\delta)$ be a $T_{0}$ metric-like space and $x_{o}$ be a limit point of $\mathcal{X}$.
Then there exists an admissible ideal $\mathcal{I}$ and a sequence in $\mathcal{X}$ so
that the sequence is $\mathcal{I}$-convergent to $x_{o}$ but is not $\mathcal{I}^{*}$-convergent
to $x_{o}$. 
\end{thm}

\begin{proof}
Since $x_{o}$ is a limit point of $\mathcal{X}$, so for each $j\in \mathbb{N}, ~\mathcal{B}_{\delta}(x_{o};\frac{1}{j})$
contains infinitely many points other than $x_{o}$. For each $j\in \mathbb{N}$ we choose $x_{j}\in \mathcal{B}_{\delta}
(x_{o};\frac{1}{j}) - \{x_{o}\}$. Then
\begin{align} \label{ eq 9}
\left|\delta(x_{j},x_{o}) - \delta(x_{o},x_{o})\right| < \frac{1}{j}, ~\forall~ j \in \mathbb{N}.
\end{align}
Let $\epsilon > 0$ be given. Then there exists $j_{o} \in \mathbb{N}$ such that $\frac{1}{j_{o}} < \epsilon$.
So from (\ref{ eq 9}) we have
\begin{align} \label{ eq 10}
& \left|\delta(x_{j},x_{o}) - \delta(x_{o},x_{o})\right| < \frac{1}{j} \leq \frac{1}{j_{o}} < \epsilon, ~\forall~ 
j \geq j_{o} \\
\Rightarrow~ & \displaystyle{\lim_{j\rightarrow \infty}} \delta(x_{j},x_{o}) = \delta(x_{o},x_{o}) \nonumber
\end{align}
Now we consider the decomposition of $\mathbb{N}$ and the ideal $\mathcal{I}$ as described in Example \ref{ Eaxmp 1}.

Now we define a sequence $\{y_{n}\}_{n \in \mathbb{N}}$ in $\mathcal{X}$ as follows:
\begin{align*}
y_{n} = x_{j}, ~\mbox{if}~ n \in \mathcal{D}_{j}, ~ j \in \mathbb{N}.
\end{align*}
Then from (\ref{ eq 10}) we have, $\forall~ n \in \mathcal{D}_{j}$ 
\begin{align*}
\left|\delta(y_{n},x_{o}) - \delta(x_{o},x_{o})\right| < \epsilon, ~\forall~ j \geq j_{o}.
\end{align*}
This implies $\{n\in \mathbb{N}: \left|\delta(y_{n},x_{o}) - \delta(x_{o},x_{o})\right| \geq \epsilon\} \subset
\bigcup^{j_{o}-1}_{j=1}\mathcal{D}_{j}$ and so $\{n\in \mathbb{N}: \left|\delta(y_{n},x_{o}) - \delta(x_{o},x_{o})\right|
\geq \epsilon\} \in \mathcal{I}$. Therefore $\mathcal{I}-\displaystyle{\lim_{n\rightarrow \infty}} \delta(y_{n},x_{o}) = 
\delta(x_{o},x_{o})$ and hence $\mathcal{I}-\displaystyle{\lim_{n\rightarrow \infty}} y_{n} = y_{o}$.

If possible, let $\mathcal{I}^{*}-\displaystyle{\lim_{n\rightarrow \infty}} x_{n} = x_{o}$. Then $\mathcal{I}^{*}-
\displaystyle{\lim_{n\rightarrow \infty}} \delta(x_{n},x_{o}) = \delta(x_{o},x_{o})$. So there exists $H \in \mathcal{I}$
such that $\mathcal{M} = \mathbb{N} - H = \{n_{1} < n_{2} < \ldots\} \in \mathcal{F(I)}$ and
\begin{align} \label{ eq 11}
\displaystyle{\lim_{n\rightarrow \infty}} \delta(y_{n_{k}},x_{o}) = \delta(x_{o},x_{o}).
\end{align}
Now $H \in \mathcal{I}$, so there exists $p \in \mathbb{N}$ such that $H \cap \mathcal{D}_{p} = \phi$. This
implies $\mathcal{D}_{p} \subset \mathcal{M}$.

Let $r_{p} = \left|\delta(x_{p},x_{o}) - \delta(x_{o},x_{o})\right|$. As $(\mathcal{X},\delta)$ is $T_{0}$,
so there must exists $\epsilon_{o} > 0$ such that $x_{p} \notin \mathcal{B}_{\delta}(x_{o};\epsilon_{o})$ and
hence $r_{p} > 0$. Again from (\ref{ eq 11}) we have, there exists $k_{o} \in \mathbb{N}$ such that
\begin{align} \label{ eq 12}
\left|\delta(y_{n_{k}},x_{o}) - \delta(x_{o},x_{o})\right| < r_{p}, ~\forall~ k \geq k_{o}.
\end{align}
As $\mathcal{D}_{p}$ is an infinite subset of $\mathbb{N}$, there must exists $k_{1} \in \mathbb{N}$ such that
$k_{1} > k_{o}$ and $n_{k_{1}} \in \mathcal{D}_{p} \cap \mathcal{M}$. So from (\ref{ eq 12}) we have
\begin{align*}
& \left|\delta(y_{n_{k_{1}}},x_{o}) - \delta(x_{o},x_{o})\right| < r_{p} \\
\Rightarrow~ & \left|\delta(x_{p},x_{o}) - \delta(x_{o},x_{o})\right| < r_{p} \\
\Rightarrow~ & r_{p} < r_{p}, ~\mbox{a contradiction}.
\end{align*}
So the sequence $\{x_{n}\}_{n\in \mathbb{N}}$ is not $\mathcal{I}^{*}$-convergence to $x_{o}$. 
\end{proof}

\begin{thm}
Let $(\mathcal{X},\delta)$ be a metric-like space and $x_{o}$ be an isolated point of $\mathcal{X}$.
If for a sequence $\{x_{n}\}_{n\in \mathbb{N}}$ in $\mathcal{X}, ~\mathcal{I}-\displaystyle{\lim
_{n\rightarrow \infty}}x_{n} = x_{o}$, then $\mathcal{I}^{*}-\displaystyle{\lim_{n\rightarrow \infty}} x_{n} = x_{o}$.
\end{thm}

\begin{proof}
Since $x_{o}$ be an isolated point of $\mathcal{X}$, so there exists $\epsilon_{o} > 0$ such that
$\mathcal{B}_{\delta}(x_{o};\epsilon_{o}) = \{x_{o}\}$. Let $\{x_{n}\}_{n\in \mathbb{N}}$ be a sequence 
in $\mathcal{X}$ such that
$\mathcal{I}-\displaystyle{\lim_{n\rightarrow \infty}} x_{n} = x_{o}$. Then
\begin{align*}
\mathcal{A}(\epsilon_{o}) = \{n\in \mathbb{N}: \left|\delta(x_{n},x_{o}) - \delta(x_{o},x_{o})\right| \geq 
\epsilon_{o}\} \in \mathcal{I}.
\end{align*}
As $\mathcal{I}$ is an admissible ideal in $\mathbb{N}$, so 
\begin{align*}
\mathcal{M} = \{n\in \mathbb{N}: \left|\delta(x_{n},x_{o}) - \delta(x_{o},x_{o})\right| < \epsilon_{o}\} = 
\{n\in \mathbb{N}: x_{n} = x_{o}\} \in \mathcal{F(I)}.
\end{align*}
Let $\mathcal{M} = \{n_{1} < n_{2} <\ldots\}$. Then
$\{x_{n_{k}}\}_{k\in \mathbb{N}} = \{x_{o}, x_{o},\ldots\}$ is a subsequence of $\{x_{n}\}_{n\in \mathbb{N}}$ and
\begin{align*}
& \displaystyle{\lim_{k\rightarrow \infty}} \delta(x_{n_{k}},x_{o}) = \delta(x_{o},x_{o}) \\
\mbox{i.e.}~ & \displaystyle{\lim_{k\rightarrow \infty}} x_{n_{k}} = x_{o}.
\end{align*} 
This implies $\mathcal{I}^{*}-\displaystyle{\lim_{n\rightarrow \infty}} x_{n} = x_{o}$.
\end{proof}

\begin{thm}
Let $(\mathcal{X},\delta)$ be a metric-like space and $\mathcal{I}$ be an admissible ideal satisfying \textit{(AP)} condition.
Then for a sequence $\{x_{n}\}_{n\in \mathbb{N}}$ in $\mathcal{X}$, if $\mathcal{I}-\displaystyle{\lim_{n\rightarrow
\infty}}x_{n} = x_{o}$, then $\mathcal{I}^{*}-\displaystyle{\lim_{n\rightarrow \infty}} x_{n} = x_{o}$.
\end{thm}

\begin{proof}
Let $\{x_{n}\}_{n\in \mathbb{N}}$ be a sequence in $\mathcal{X}$ and $x_{o} \in \mathcal{X}$ such that
$\mathcal{I}-\displaystyle{\lim_{n\rightarrow \infty}} x_{n} = x_{o}$. Then $\displaystyle{\lim_{n\rightarrow
\infty}} \delta(x_{n},x_{o}) = \delta(x_{o},x_{o})$. So for every $\eta > 0$, the set
$\mathcal{A}(\eta) = \{n\in \mathbb{N}: \left|\delta(x_{n},x_{o}) - \delta(x_{o},x_{o})\right| \geq \eta\}
\in \mathcal{I}$.

Let $A_{1} = \{n\in \mathbb{N}: \left|\delta(x_{n},x_{o}) - \delta(x_{o},x_{o})\right| \geq 1\}$ and for each 
$j\in \mathbb{N} ~\mbox{with}~ j \geq 2, ~A_{j} = \{n\in \mathbb{N}: \frac{1}{j} \leq \left|\delta(x_{n},x_{o}) - 
\delta(x_{o},x_{o})\right| < \frac{1}{j-1}\}$. Then $\{A_{j}\}_{j\in \mathbb{N}}$ is a countable family of 
mutually disjoint sets belonging to $\mathcal{I}$. As $\mathcal{I}$ satisfies \textit{(AP)} condition, 
there exists a countable collection of sets $\{B_{j}\}_{j\in \mathbb{N}}$ such that
$\left|A_{j} \Delta B_{j}\right| < \infty, ~\forall~ j \in \mathbb{N}$ and
$B = \bigcup^{\infty}_{j=1} B_{j} \in \mathcal{I}$. Let $\mathcal{M} = \mathbb{N} - B$. As $\mathcal{I}$ is an 
admissible ideal, so we have $\mathcal{M} \in \mathcal{F(I)}$ and $\mathcal{M}$ is an infinite subset of $\mathbb{N}$,
say $\mathcal{M} = \{n_{1} < n_{2} < \ldots\}$.

Let $\epsilon > 0$ be given. Then there exists $j_{o} \in \mathbb{N}$ such that $\frac{1}{j_{o}} < \epsilon$.
So for all $j \geq (j_{o} + 1)$ we have
\begin{align*}
& \left|\delta(x_{n},x_{o}) - \delta(x_{o},x_{o})\right| < \frac{1}{j-1} \leq \frac{1}{j_{o}} < \epsilon, ~\forall~
n \in A_{j} \\
\Rightarrow~ & \bigcup^{\infty}_{j=j_{o}+1} A_{j} \subset \{n \in \mathbb{N}: \left|\delta(x_{n},x_{o}) - \delta(x_{o},x_{o})
\right| < \epsilon\}.
\end{align*}
So, $\{n\in \mathbb{N}: \left|\delta(x_{n},x_{o}) - \delta(x_{o},x_{o})\right| \geq \epsilon\} \subset \bigcup
^{j_{o}}_{j=1} A_{j}$.

Now $\left|A_{j} \Delta B_{j}\right| < \infty$ for all $j = 1, 2, \ldots, j_{o}$. So there exists $n_{o} \in \mathbb{N}$
such that for all $n \in \mathbb{N}$ with $n > n_{o}$, we have $n \notin \bigcup^{j_{o}}_{j=1} (A_{j} \Delta B_{j})$.
Therefore if $n \in \mathbb{N}$ and $n > n_{o}$, then
\begin{align} \label{ eq 13}
n \in \bigcup^{j_{o}}_{j=1} B_{j} \Leftrightarrow n \in \bigcup^{j_{o}}_{j=1} A_{j}.
\end{align}

Let $k_{o} = k_{o}(\epsilon) \in \mathbb{N}$ be such that $n_{k} > n_{o}$ if $k \geq k_{o}$. Then for all
$k \geq k_{o}$ we have $n_{k} \notin B$ and so $n_{k} \notin \bigcup^{j_{o}}_{j=1} B_{j}$. Then from
(\ref{ eq 13}) we have $n_{k} \notin \bigcup^{j_{o}}_{j=1} A_{j}$. So we must have 
\begin{align*}
& \left|\delta(x_{n_{k}},x_{o}) - \delta(x_{o},x_{o})\right| < \frac{1}{j_{o}} < \epsilon \\
\mbox{i.e.}~ & \left|\delta(x_{n_{k}},x_{o}) - \delta(x_{o},x_{o})\right| < \epsilon, ~\forall~ k \geq k_{o}.
\end{align*}
This implies $\mathcal{I}^{*}-\displaystyle{\lim_{n\rightarrow \infty}} x_{n} = x_{o}$.
\end{proof}

\begin{thm}
If $x_{o}$ is a limit point of a $T_{0}$ metric-like space $(\mathcal{X},\delta)$ and if for any sequence $\{x_{n}\}_{n\in \mathbb{N}}$
in $\mathcal{X}, ~\mathcal{I}-\displaystyle{\lim_{n\rightarrow \infty}} x_{n} = x_{o}$ implies $\mathcal{I}^{*}-\displaystyle
{\lim_{n\rightarrow \infty}} x_{n} = x_{o}$, then $\mathcal{I}$ satisfies the property \textit{(AP)}.
\end{thm}

\begin{proof}
Proof is straightforward, so is omitted
\end{proof}


{\bf Acknowledgement : } The second author is grateful to the University Grants Commission, India for
financial support under UGC-JRF scheme during the preparation of this paper.


\end{document}